\documentclass[11pt]{amsart}

\usepackage[foot]{amsaddr}
\usepackage[english]{babel}

\allowdisplaybreaks

\usepackage[
backend=biber,
style=numeric,
giveninits=false,   
maxbibnames=99,     
minbibnames=99,
maxcitenames=99,
mincitenames=99,
doi=true,
url=false,
isbn=false
]{biblatex}

\DeclareFieldFormat[article]{title}{#1}

\DeclareFieldFormat{journaltitle}{\mkbibemph{#1}}

\DeclareFieldFormat[article]{volume}{\mkbibbold{#1}}

\DeclareFieldFormat{pages}{#1}

\DeclareFieldFormat{eid}{Article number\space #1}

\DeclareFieldFormat{doi}{\mkbibacro{DOI}\addcolon\space\url{https://doi.org/#1}}

\DeclareBibliographyDriver{article}{%
	\usebibmacro{bibindex}%
	\usebibmacro{begentry}%
	\printnames{author}%
	\setunit{\addperiod\space}%
	\printfield{title}%
	\setunit{\addperiod\space}%
	\printfield{journaltitle}%
	\setunit{\addspace}%
	\printfield{volume}%
	\setunit{\addspace}%
	\printtext[parens]{\printdate}%
	\setunit{\addspace}%
	\iffieldundef{pages}
	{\printfield{eid}}
	{\printfield{pages}}%
	\newunit\newblock
	\printfield{doi}%
	\addperiod%
	\usebibmacro{finentry}%
}
\renewbibmacro*{in:}{}

\usepackage[
pdftex,hyperfootnotes]{hyperref}
\hypersetup{
 colorlinks=true,
 linkcolor=NavyBlue, 
 urlcolor=RoyalPurple,
 citecolor=OliveGreen,
 pdftitle={Unbounded banded matrices and mixed-type multiple orthogonality},
 bookmarks=true,
}

\usepackage{doi}
\usepackage{xurl}

\usepackage[usenames,dvipsnames,svgnames,table,x11names]{xcolor}
\usepackage{pgfplots}
\usepackage{tikz}
\usepackage{tikz-3dplot}
\usetikzlibrary{automata,quotes, chains,matrix,calc,shadows,shapes.callouts,shapes.geometric,shapes.misc,positioning,patterns,decorations.shapes,
 decorations.pathmorphing,decorations.markings,decorations.fractals,decorations.pathreplacing,shadings,fadings,arrows.meta,bending}

\usepackage{nicematrix,drawmatrix}
\usepackage[utf8]{inputenc}

\usepackage{comment}
\usepackage[textwidth=17.5cm,textheight=22.275cm,
height=
24.275cm,
width=18cm]{geometry}

\usepackage{amssymb,latexsym,amsmath,amsthm,bm}
\usepackage{mathrsfs}
\usepackage{mathtools,arydshln,mathdots}
\mathtoolsset{showonlyrefs}
\usepackage{tcolorbox}

\usepackage[table]{xcolor}

\usepackage{Baskervaldx}
\usepackage[]{newtxmath}

\usepackage{enumerate}
\usepackage{bigints}

\theoremstyle{plain}

\newtheorem{teo}{Theorem}[section]

\newtheorem{lemma}[teo]{Lemma}

\renewcommand{\d}{\operatorname{d}}

\newcommand{\diag}{\operatorname{diag}}

\newcommand{\C}{\mathbb{C}}

\newcommand{\N}{\mathbb{N}}
\newcommand{\R}{\mathbb{R}}

\allowdisplaybreaks[1]

\makeatletter
\DeclareRobustCommand{\gaussk}{\DOTSB\gaussk@\slimits@}
\newcommand{\gaussk@}{\mathop{\vphantom{\sum}\mathpalette\bigcal@{K}}}

\newcommand{\bigcal@}[2]{%
 \vcenter{\m@th
  \sbox\z@{$#1\sum$}%
  \dimen@=\dimexpr\ht\z@+\dp\z@
  \hbox{\resizebox{!}{0.8\dimen@}{$\mathcal{K}$}}%
 }%
}
\newcommand{\cfracplus}{\mathbin{\cfracplus@}}
\newcommand{\cfracplus@}{%
 \sbox\z@{$\dfrac{1}{1}$}%
 \sbox\tw@{$+$}%
 \raisebox{\dimexpr\dp\tw@-\dp\z@\relax}{$+$}%
}
\newcommand{\cfracdots}{\mathord{\cfracdots@}}
\newcommand{\cfracdots@}{%
 \sbox\z@{$\dfrac{1}{1}$}%
 \sbox\tw@{$+$}%
 \raisebox{\dimexpr\dp\tw@-\dp\z@\relax}{$\cdots$}%
}
\makeatother

\makeatletter
\newcommand*{\relrelbarsep}{.386ex}
\newcommand*{\relrelbar}{%
 \mathrel{%
  \mathpalette\@relrelbar\relrelbarsep
 }%
}
\newcommand*{\@relrelbar}[2]{%
 \raise#2\hbox to 0pt{$\m@th#1\relbar$\hss}%
 \lower#2\hbox{$\m@th#1\relbar$}%
}
\providecommand*{\rightrightarrowsfill@}{%
 \arrowfill@\relrelbar\relrelbar\rightrightarrows
}
\providecommand*{\leftleftarrowsfill@}{%
 \arrowfill@\leftleftarrows\relrelbar\relrelbar
}
\providecommand*{\xrightrightarrows}[2][]{%
 \ext@arrow 0359\rightrightarrowsfill@{#1}{#2}%
}
\providecommand*{\xleftleftarrows}[2][]{%
 \ext@arrow 3095\leftleftarrowsfill@{#1}{#2}%
}
\makeatother

\catcode`,\active

\catcode`\,12

\usepackage{appendix}

\usepackage{xcolor} 

\usepackage[normalem]{ulem}
\usepackage{soul} 

\usepackage{comment} 

\usepackage{tikz}
\usetikzlibrary{shapes,arrows}
\usepackage{verbatim}

\tikzstyle{block} = [draw, rectangle, 
minimum height=3em, minimum width=2em]

\usepackage{mathrsfs} 

\begin{document}

\title[Unbounded banded matrices and mixed-type multiple orthogonality]
{Unbounded banded matrices, shifted positive bidiagonal factorizations, and mixed-type multiple orthogonality}

\author[A Branquinho]{Amílcar Branquinho$^{1}$}
\address{$^1$Departamento de Matemática,
 Universidade de Coimbra, 3001-454 Coimbra, Portugal}
\email{ajplb@mat.uc.pt}

\author[A Foulquié-Moreno]{Ana Foulquié-Moreno$^{2}$}
\address{$^2$Departamento de Matemática, Universidade de Aveiro, 3810-193 Aveiro, Portugal}
\email{foulquie@ua.pt}

\author[M Mañas]{Manuel Mañas$^{3}$}
\address{$^3$Departamento de Física Teórica, Universidad Complutense de Madrid, Plaza Ciencias 1, 28040-Madrid, Spain}
\email{manuel.manas@ucm.es}



\enlargethispage{1.25cm}

\begin{abstract}
This work extends Favard-type spectral representations for banded matrices $T$ beyond the bounded setting. It assumes that, for every $N\in\N_0$, there exists a shift $s_N\ge 0$ such that the shifted truncation $A_N\coloneq T^{[N]}+s_N I_{N+1}$ admits a positive bidiagonal factorization (PBF). Allowing $s_N$ to depend on $N$ leads to a natural recentering step: the discrete Gauss-type quadrature measures associated with $A_N$ are translated by $x\mapsto x-s_N$, producing a uniformly bounded family of distribution functions. Combining moment stabilization for banded truncations with Helly-type compactness theorems yields a limiting matrix-valued measure, together with a Favard-type spectral representation and the corresponding mixed-type multiple biorthogonality relations. As a consequence, the classical Favard theorem for (possibly unbounded) Jacobi matrices is recovered as a special case. Indeed, for a tridiagonal $J$ with positive sub- and superdiagonals, each truncation $J^{[N]}$ admits a shift $s_N\ge 0$ such that $J^{[N]}+s_N I_{N+1}$ is oscillatory and therefore admits a PBF. The preceding construction then produces the usual spectral measure for $J$.
\end{abstract}
 
 \maketitle
 
 

\thispagestyle{empty}
 \section{Introduction}
The spectral Favard theorem for banded matrices admitting a positive bidiagonal factorization
was established in the bounded case in \cite{BFM1}. In that work, boundedness makes it possible
to construct a matrix-valued spectral measure and to derive a Gaussian quadrature formula,
which becomes a central tool in the proof. Further structural properties, including total
positivity and the normality (maximal degree pattern) of the associated mixed multiple
orthogonal polynomials, were analyzed in \cite{BFM4}, still under the boundedness hypothesis.

The purpose of the present paper is to remove the boundedness assumption. In the unbounded setting,
the positive bidiagonal factorization cannot, in general, be imposed directly at the level of the
semi-infinite matrix. Accordingly, we work with finite truncations and assume that for every
$N\in\N_0$ there exists a shift $s_N\ge 0$ such that the shifted truncation $T^{[N]}+s_N I_{N+1}$
admits a positive bidiagonal factorization. Under this hypothesis, and keeping the banded structure
of $T$, we prove that the spectral Favard theorem remains valid in the unbounded case. The argument
builds on the constructive framework introduced in \cite{BFM1} and follows the same discrete-to-limit
philosophy: one constructs discrete matrix-valued measures at the truncation level, and then passes
to a limiting measure as the truncation size tends to infinity.

More specifically, for each finite truncation one relies on the positive bidiagonal factorization,
on the spectral properties of oscillatory matrices, and on the sign-variation structure of their
eigenvectors \cite{Gantmacher,Fallat}. In the banded setting this yields positivity of the associated
Christoffel numbers and Gauss-type quadrature formulas at the discrete level \cite{BFM1}. The main
difficulty in the unbounded regime is that the shifts $s_N$ may depend on $N$, so the discrete measures
must be recentered before taking limits. Following the strategy developed by Chihara for orthogonal
polynomials \cite[Chap. 2]{Chihara} and using a Helly-type selection principle, we combine the discrete
quadrature formulas with uniform boundedness of the translated distribution functions to extract a
convergent subsequence of measures and identify its limit. This provides the matrix-valued spectral
measure and completes the proof of the spectral Favard theorem without boundedness.

A further point concerns the degree pattern of the mixed multiple orthogonal polynomials. As shown
in \cite{BFM4}, normality (maximal degrees along the step-line) is established before taking the
large-$N$ limit. In both the bounded and unbounded settings, the required initial conditions must be
constructed from the positive bidiagonal factorization. There is a natural freedom in this choice,
which can be parametrized by families of positive triangular matrices, or more restrictively by totally
positive triangular matrices (all nontrivial minors are positive). The latter choice guarantees that maximal degrees are attained, exactly as in the normality result proved in~\cite{BFM4}.

Totally nonnegative and totally positive matrices constitute a fundamental class of operators with deep
connections to oscillation theory, approximation theory, and spectral analysis. Their basic properties
were established in the seminal work of Gantmacher and Krein \cite{Gantmacher}, and have since been
developed extensively, see for example the monographs of Karlin \cite{Karlin} and Pinkus \cite{Pinkus},
as well as the modern reference by Fallat and Johnson \cite{Fallat}. Of particular relevance in the
present context are oscillatory matrices, whose spectra consist of simple positive eigenvalues and
whose eigenvectors exhibit a precise sign variation structure. These features make totally nonnegative
matrices especially well adapted to the spectral analysis of banded operators admitting positive
bidiagonal factorizations, and provide the key mechanism ensuring the positivity of Christoffel-type
coefficients and the normality of the associated families of multiple orthogonal polynomials.

Multiple orthogonal polynomials arise as natural generalizations of classical orthogonal polynomials
when orthogonality conditions are imposed with respect to several measures. They play a central role in
rational approximation theory, number theory, random matrix models, and integrable systems. A systematic
account of the theory can be found in the monograph of Nikishin and Sorokin \cite{nikishin_sorokin}, as
well as in the classical treatment of Ismail \cite{Ismail}. From a modern perspective, multiple orthogonal
polynomials are most naturally characterized through recurrence relations associated with banded operators,
leading to matrix-valued and vector-valued spectral problems. In this framework, polynomials of mixed type
emerge as the appropriate objects encoding the joint left--right spectral structure of banded matrices, see
for instance \cite{andrei_walter,afm}. The mixed-type setting is particularly well suited for the study of
non-symmetric banded operators, where biorthogonality and matrix-valued measures replace the scalar
orthogonality of the classical theory.
 
\subsection{Banded matrices and mixed-type multiple orthogonal polynomials}

Let us define the recursion polynomials associated with the banded matrix
\begin{align}\label{eq:monic_Hessenberg}
 T&=
 \left[ \begin{NiceMatrix}
  T_{0,0} &\Cdots &&T_{0,q}& 0 & \Cdots &&\\
  \Vdots& &&&\Ddots &\Ddots& &\\
  T_{p,0}&&&&&&&\\
  0&\Ddots[shorten-end=15pt]&&&&&&\\
  \Vdots[shorten-end=5pt]&\Ddots[shorten-end=5pt]&&&&&&\\
  &&&&&&&\\
  &&&&&&&\\[6pt]
  &&&&&& &
 \end{NiceMatrix}\right]
\end{align}
as the entries of semi-infinite left and right eigenvectors,
\begin{align*}
\begin{aligned}
 A^{(a)}&=\left[\begin{NiceMatrix}
  A^{(a)}_0 &  A^{(a)}_1& \Cdots
 \end{NiceMatrix}\right], & a&\in\{1,\dots, p\},&
 B^{(b)}&=\begin{bNiceMatrix}
  B^{(b)}_0 \\[2pt] B^{(b)}_1\\ \Vdots
 \end{bNiceMatrix}, & b&\in\{1,\dots, q\},
\end{aligned}
\end{align*}
which satisfy the eigenvalue relations
\begin{align*}
\begin{aligned}
 A^{(a)}T&=xA^{(a)},& a&\in\{1,\dots, p\},&
 TB^{(b)}&=xB^{(b)}, &b&\in\{1,\dots, q\}.
\end{aligned}
\end{align*}

Throughout we assume that the semi-infinite matrix $T$ admits a normalized positive bidiagonal factorization (PBF),
\begin{align}\label{eq:bidiagonal}
 T= L_{1} \cdots L_{p} \Delta U_q\cdots U_1,
\end{align}
where $\Delta=\diag(\Delta_0,\Delta_1,\dots)$ and the bidiagonal factors are normalized (their diagonal entries equal $1$) and given by
\begin{align}\label{eq:bidiagonal_factors}
 \begin{aligned}
  L_k&\coloneq \left[\begin{NiceMatrix}[columns-width=auto]
   1 &0&\Cdots&\\
   L_{k|1,0} & 1 &\Ddots&\\
   0& L_{k|2,1}& 1& \\
   \Vdots[shorten-start=5pt,shorten-end=7pt] &\Ddots& \Ddots& \Ddots\\&&&
  \end{NiceMatrix}\right], & 
  U_k& \coloneq
  \left[\begin{NiceMatrix}[columns-width = auto]
   1& U_{k|0,1}&0&\Cdots&\\
   0& 1& U_{k|1,2}&\Ddots&\\
   \Vdots[shorten-end=5pt]&\Ddots[shorten-end=17pt]&1&\Ddots&\\
   & &\Ddots[shorten-end=12pt] &\Ddots[shorten-end=8pt] &\\&&&&
  \end{NiceMatrix}\right], 
 \end{aligned}
\end{align}
so that the positivity constraints $L_{k|i,j},U_{k|i,j},\Delta_i>0$, for $i\in\N_0$, hold.

The components of the eigenvectors are polynomials in the spectral parameter \(x\); we refer to them as the left and right recursion polynomials. They are fixed by the initial conditions
\begin{align}
 \label{eq:initcondtypeI}
 \begin{aligned}
  \begin{cases}
   A^{(1)}_0=1 , \\
   A^{(1)}_1= \nu^{(1)}_1 , \\
   \hspace{.895cm} \vdots \\
   A^{(1)}_{p-1}=\nu^{(1)}_{p-1} ,
  \end{cases}
  &&
  \begin{cases}
   A^{(2)}_0=0 , \\
   A^{(2)}_1= 1 , \\
   A^{(2)}_2= \nu^{(2)}_2 , \\
   \hspace{.895cm} \vdots \\
   A^{(2)}_{p-1}=\nu^{(2)}_{p-1} ,
  \end{cases}
  && \cdots &&
  \begin{cases}
   A^{(p)}_0 =0 , \\
   \hspace{.915cm} \vdots \\
   A^{(p)}_{p-2} = 0 , \\
   A^{(p)}_{p-1} = 1,
  \end{cases}
 \end{aligned}
\end{align}
with \( \nu^{(i)}_{j} \) arbitrary constants, and
\begin{align}
 \label{eq:initcondtypeII}
 \begin{aligned}
  \begin{cases}
   B^{(1)}_0=1 , \\
   B^{(1)}_1= \xi^{(1)}_1 , \\
   \hspace{.895cm} \vdots \\
   B^{(1)}_{q-1}=\xi^{(1)}_{q-1} ,
  \end{cases}
  &&
  \begin{cases}
   B^{(2)}_0=0 , \\
   B^{(2)}_1= 1 , \\
   B^{(2)}_2= \xi^{(2)}_2 , \\
   \hspace{.895cm} \vdots \\
   B^{(2)}_{q-1}=\xi^{(2)}_{q-1} ,
  \end{cases}
  && \cdots &&
  \begin{cases}
   B^{(q)}_0 =0 , \\
   \hspace{.915cm} \vdots \\
   B^{(q)}_{q-2} = 0 , \\
   B^{(q)}_{q-1} = 1,
  \end{cases}
 \end{aligned}
\end{align}
where \( \xi^{(i)}_{j} \) are also arbitrary. We collect these parameters into the initial condition matrices
\begin{align}
 \label{eq:ic}
 \begin{aligned}
  \nu&\coloneq \begin{bNiceMatrix}
   1& 0 & \Cdots& && 0 \\
   \nu^{(1)}_1 & 1 & \Ddots&& & \Vdots \\
   \Vdots & \Ddots[shorten-start=-3pt,shorten-end=-3pt] & \Ddots& && \\
   &&&& &\\&&&&&0\\
   \nu^{(1)}_{p-1} &\Cdots& && \nu^{(p-1)}_{p-1}& 1 
  \end{bNiceMatrix} ,&
  \xi&\coloneq \begin{bNiceMatrix}
   1& 0 & \Cdots& && 0 \\
   \xi^{(1)}_1 & 1 & \Ddots&& & \Vdots \\
   \Vdots & \Ddots[shorten-start=-3pt,shorten-end=-3pt] & \Ddots& && \\
   &&&& &\\&&&&&0\\
   \xi^{(1)}_{q-1} &\Cdots& && \xi^{(q-1)}_{q-1}& 1 
  \end{bNiceMatrix}.
 \end{aligned}
\end{align}

The recursion polynomials are uniquely determined by \eqref{eq:initcondtypeI}--\eqref{eq:initcondtypeII} together with the banded recursions
\begin{align}\label{eq:recursion_dual_A}
 A^{(a)}_{n-q} T_{n-q,n}+ \cdots +A^{(a)}_{n+p} T_{n+p,n}&= x A^{(a)}_{n}, & n &\in\{0,1,\ldots\}, & a &\in\{1,\dots, p\}, & A_{-q}^{(a)} &=\dots=A^{(a)}_{-1}= 0,\\
 \label{eq:recursion_B}
 T_{n,n-p}B^{(b)}_{n-p} + \cdots + T_{n,n+q}B^{(b)}_{n+q} &= x B^{(b)}_{n}, & n &\in\{0,1,\ldots\}, & b&\in\{1,\dots, q\}, & B_{-p}^{(b)} &=\dots=B^{(b)}_{-1}= 0.
\end{align}

For the semi-infinite matrix \(T\) we denote by \(P_N(x)\) the characteristic polynomial of the truncation \(T^{[N-1]}\),
namely
\begin{align*}
 P_{N}(x)&\coloneq\begin{cases}
  1, & N=0,\\\det\big(xI_N-T^{[N-1]}\big), & N\in\N.
 \end{cases}
\end{align*}

Introduce the blocks of left and right recursion polynomials
\begin{align*}
\begin{aligned}
 A_N &\coloneq \begin{bNiceMatrix}
  A^{(1)}_N& \Cdots & A^{(1)}_{N+p-1}  \\[2pt]
  \Vdots & & \Vdots \\[2pt]
  A^{(p)}_N & \Cdots & A^{(p)}_{N+p-1}
 \end{bNiceMatrix}, & 
 B_N &\coloneq \begin{bNiceMatrix}
  B^{(1)}_N & \Cdots & B^{(q)}_N \\[2pt]
  \Vdots & & \Vdots \\[2pt]
  B^{(1)}_{N+q-1} & \Cdots & B^{(q)}_{N+q-1}
 \end{bNiceMatrix},& N&\in \N_0 ,
\end{aligned}
\end{align*}
and the products
\begin{align*}
\begin{aligned}
 \alpha_N &\coloneq (-1)^{(p-1)N}T_{p,0}\cdots T_{N+p-1,N-1}, &
 \beta_N &\coloneq (-1)^{(q-1)N}T_{0,q}\cdots T_{N-1,N+q-1} , & N&\in \N,
\end{aligned}
\end{align*}
with \( \alpha_0=\beta_0=1 \). Since the entries on the extreme diagonals are nonzero, we have
\( \alpha_N,\beta_N\neq 0 \).
For \( N\in\N_0 \) these quantities relate the characteristic polynomials and the recursion blocks through
\begin{align*}
 P_N(x)&=\alpha_N\det A_N(x)=\beta_N\det B_N(x).
\end{align*}

Next we introduce determinantal polynomials, built from the left and right recursion polynomials, which produce left and right eigenvectors of $T^{[N]}$. Define
\begin{align}\label{eq:QNn}
 \begin{aligned}
  Q_{n,N}&\coloneq\begin{vNiceMatrix}
   A^{(1)}_{n} & \Cdots[shorten-start=-3pt] & A^{(p)}_{n} \\[2pt]
   A^{(1)}_{N+1} & \Cdots[shorten-start=-1pt] & A^{(p)}_{N+1} \\[2pt]
   \Vdots & & \Vdots \\[2pt]
   A^{(1)}_{N+p-1} & \Cdots[shorten-start=-3pt] & A^{(p)}_{N+p-1}
  \end{vNiceMatrix},&
  R_{n,N}&\coloneq\begin{vNiceMatrix}
   B^{(1)}_{n} & \Cdots[shorten-start=-3pt] & B^{(q)}_{n} \\[2pt]
   B^{(1)}_{N+1} & \Cdots[shorten-start=-1pt] & B^{(q)}_{N+1} \\[2pt]
   \Vdots & & \Vdots \\[2pt]
   B^{(1)}_{N+q-1} & \Cdots[shorten-start=-3pt] & B^{(q)}_{N+q-1}
  \end{vNiceMatrix},
 \end{aligned}
\end{align}
together with the semi-infinite row and column vectors
\begin{align*}
\begin{aligned}
 Q_N&\coloneq\left[\begin{NiceMatrix}
  Q_{0,N} &Q_{1,N} &\Cdots
 \end{NiceMatrix}\right], &
 R_N&\coloneq\left[\begin{NiceMatrix}
  R_{0,N} \\R_{1,N} \\\Vdots
 \end{NiceMatrix}\right],
\end{aligned}
\end{align*}
and their truncations
\begin{align*}
\begin{aligned}
 Q^{\langle N\rangle}&\coloneq \begin{bNiceMatrix}
  Q_{0,N} &Q_{1,N}&\Cdots & Q_{N,N}
 \end{bNiceMatrix}, & 
 R^{\langle N\rangle}&\coloneq \begin{bNiceMatrix}
  R_{0,N} \\R_{1,N}\\\Vdots \\ R_{N,N}
 \end{bNiceMatrix}.
\end{aligned}
\end{align*}

A generalized Christoffel--Darboux formalism for banded matrices yields the following identities.
\begin{enumerate}[\rm i)]
\item The biorthogonal families of left and right eigenvectors $\big\{w^{\langle N\rangle}_k\big\}_{k=1}^{N+1}$ and $\big\{u^{\langle N\rangle}_k\big\}_{k=1}^{N+1}$ can be expressed as
 \begin{align}\label{eq:eigenvectors}
  \begin{aligned}
   w^{\langle N\rangle}_{k}&=\frac{ Q^{\langle N\rangle}\big(\lambda^{[N]}_k\big)}{\beta_N\sum_{l=0}^{N}Q_{l,N}\big(\lambda^{[N]}_k\big)R_{l,N}\big(\lambda^{[N]}_k\big)}, &
   u^{\langle N\rangle}_{k}&=\beta_N R^{\langle N\rangle}\big(\lambda^{[N]}_k\big).
  \end{aligned}
 \end{align}
In particular, $w^{\langle N\rangle}_{k}u^{\langle N\rangle}_{l}=\delta_{k,l}$.
 
 \item In terms of the characteristic polynomial, the entries of the left eigenvectors admit the representation
\begin{align*}
 \begin{aligned}
  w^{\langle N\rangle}_{k,n}&=
  \frac{ \alpha_N Q_{n,N}\big(\lambda^{[N]}_k\big)
  }{
   P_{N}\big(\lambda^{[N]}_k\big)P'_{N+1}\big(\lambda^{[N]}_k\big)}.
 \end{aligned}
\end{align*}
 
 \item Let $\mathscr U$ be the matrix whose columns are the right eigenvectors $u_k$ in the standard order, and let $\mathscr W$ be the matrix whose rows are the left eigenvectors $w_k$ in the standard order. Then
 $ \mathscr U\mathscr W=\mathscr W\mathscr U=I_{N+1}$.
 
 \item Writing $D=\diag\big(\lambda^{[N]}_1,\lambda^{[N]}_2,\dots,\lambda^{[N]}_{N+1}\big)$, one has the spectral decomposition
 \begin{align}\label{eq:UDnW=Jn}
  \begin{aligned}
   \mathscr UD^n\mathscr W&=\big(T^{[N]}\big)^n, & n&\in\N_0.
  \end{aligned}
 \end{align}
\end{enumerate}

The so-called Christoffel numbers $\mu^{[N]}_{k,a}$ and $\rho^{[N]}_{k,b}$, $a\in\{1,\dots,p\}$ and $b\in\{1,\dots,q\}$, see \cite{BFM1}, provide the expansions
\begin{align*}
\begin{aligned}
 w^{\langle N\rangle}_{k,n}&= A_{n-1}^{(1)}\big(\lambda^{[N]}_k\big)\mu^{[N]}_{k,1} +\cdots +A_{n-1}^{(p)}\big(\lambda^{[N]}_k\big)\mu^{[N]}_{k,p},\\
 u^{\langle N\rangle}_{k,n}&= B_{n-1}^{(1)}\big(\lambda^{[N]}_k\big)\rho^{[N]}_{k,1} +\cdots +B_{n-1}^{(q)}\big(\lambda^{[N]}_k\big)\rho^{[N]}_{k,q}.
\end{aligned}
\end{align*}
Moreover, these coefficients are related to the eigenvector entries and the initial conditions by
\begin{align*}
 \begin{aligned}
  \begin{bNiceMatrix}
   \mu^{[N]}_{k,1} \\[5pt]
   \mu^{[N]}_{k,2} \\
   \Vdots
   \\
   \mu^{[N]}_{k,p}
  \end{bNiceMatrix}
  &= \nu^{-1} \begin{bNiceMatrix}
   w^{\langle N\rangle}_{k,1} \\[5pt]
   w^{\langle N\rangle}_{k,2} \\
   \Vdots \\
   w^{\langle N\rangle}_{k,p}
  \end{bNiceMatrix}, &
  \begin{bNiceMatrix}
   \rho^{[N]}_{k,1} \\[5pt]
   \rho^{[N]}_{k,2} \\
   \Vdots
   \\
   \rho^{[N]}_{k,q}
  \end{bNiceMatrix}
  &= \xi^{-1} \begin{bNiceMatrix}
   u^{\langle N\rangle}_{k,1} \\[5pt]
   u^{\langle N\rangle}_{k,2} \\
   \Vdots \\
   u^{\langle N\rangle}_{k,q}
  \end{bNiceMatrix}.
 \end{aligned}
\end{align*}

As shown in \cite{BFM1}, for a suitable choice of initial conditions in terms of the bidiagonal factorization the Christoffel numbers are strictly positive, namely
\begin{align*}
\begin{aligned}
 \rho^{[N]}_{k,b}&>0, &\mu^{[N]}_{k,a}&>0, &k&\in\{1,\dots,N+1\}, &a&\in\{1,\dots, p\}, & b&\in\{1,\dots,q\}.
\end{aligned}
\end{align*}

Consider the associated step functions
\begin{align*}
\psi^{[N]}_{b,a}\coloneq \begin{cases}
 0, & x<\lambda^{[N]}_{N+1},\\[2pt]
 \rho^{[N]}_{1,b}\mu^{[N]}_{1,a}+\cdots+\rho^{[N]}_{k,b}\mu^{[N]}_{k,a}, & \lambda^{[N]}_{k+1}\leqslant x< \lambda^{[N]}_{k}, \quad k\in\{1,\dots,N\},\\[2pt]
 \rho^{[N]}_{1,b}\mu^{[N]}_{1,a}+\cdots+\rho^{[N]}_{N+1,b}\mu^{[N]}_{N+1,a} ,
 & x \geqslant \lambda^{[N]}_{1}.
\end{cases}
\end{align*}
The final step is uniformly bounded: for $a\in\{1, \dots, p\}$ and $b\in\{1, \dots, q\}$ one has
\begin{align*}
\rho^{[N]}_{1,b}\mu^{[N]}_{1,a}+\cdots+\rho^{[N]}_{N+1,b}\mu^{[N]}_{N+1,a} = (\xi^{-1}I_{q,p}\nu^{-\top})_{b,a}.
\end{align*}
Here $I_{q,p}\in\R^{q\times p}$ denotes the rectangular identity matrix, with $(I_{q,p})_{k,l}=\delta_{k,l}$.

Define the $q\times p$ matrix
\begin{align*}
\Psi^{[N]}\coloneq\begin{bNiceMatrix}[small]
 \psi^{[N]}_{1,1}&\Cdots &\psi^{[N]}_{1,p}\\
 \Vdots & &\Vdots\\
 \psi^{[N]}_{q,1}&\Cdots &\psi^{[N]}_{q,p}
\end{bNiceMatrix}
\end{align*}
and the corresponding matrix of discrete Lebesgue--Stieltjes measures supported at the zeros of $P_{N+1}$,
\begin{align*}
 \d\Psi^{[N]}=\begin{bNiceMatrix}
  \d\psi ^{[N]}_{1,1}&\Cdots & \d\psi^{[N]}_{1,p}\\
  \Vdots & &\Vdots\\
  \d\psi^{[N]}_{q,1}&\Cdots & \d\psi^{[N]}_{q,p}
 \end{bNiceMatrix}=\sum_{k=1}^{N+1}\begin{bNiceMatrix}
  \rho^{[N]}_{k,1}\\\Vdots \\ \rho^{[N]}_{k,q}
 \end{bNiceMatrix}\begin{bNiceMatrix}
  \mu^{[N]}_{k,1} & \Cdots & \mu^{[N]}_{k,p}
 \end{bNiceMatrix}\delta\big(x-\lambda^{[N]}_k\big).
\end{align*}

The following biorthogonality relations hold:
\[\begin{aligned}
 \sum_{a=1}^p\sum_{b=1}^{q} \int B^{(b)}_n(x)\d\psi^{[N]}_{b,a}(x)A^{(a)}_{m}(x)&=\delta_{n,m}, &n,m&\in\{0,\dots,N\}.
\end{aligned}\]
Moreover, for $m \in\{1,\dots, N\}$ the orthogonality relations
\[\begin{aligned}
 \sum_{a=1}^p\int x^n\d\psi ^{[N]}_{b,a}(x)A^{(a)}_{m}(x)&=0, & 
 n&\in\left\{0,\dots,\left \lceil \frac{m+1-b}{q}\right \rceil -1\right\},& b&\in\{1,\dots,q\},\\
 \sum_{b=1}^q\int B^{(b)}_{m}(x)\d\psi ^{[N]}_{b,a}(x)x^n&=0, 
 & n&\in\left\{0,\dots,\left \lceil \frac{m+1-a}{p}\right \rceil -1\right\}, & a&\in\{1,\dots,p\}.
\end{aligned}\]
are satisfied. Therefore, the recursion polynomials form discrete multiple orthogonal polynomials of mixed type on the step-line. Finally, one obtains the spectral representation
\begin{align*}
\begin{aligned}
 (T^k)_{m,n}&= \sum_{a=1}^p\sum_{b=1}^{q} \int B^{(b)}_n(x)x^k\d\psi^{[N]}_{b,a}(x)A^{(a)}_{m}(x),
 && m,n\in\{0,\dots,N\}.
\end{aligned}
\end{align*}
In the study of the quadrature formulas in \cite{BFM1}, we consider
\begin{equation}
	\label{eq:moments_discrete}\begin{aligned}
\int x^n\d\psi^{[N]} _{b,a}(x) 
&=
\sum_{k=1}^{N+1}\rho_{k,b}^{[N]}\mu_{k,a}^{[N]}\big( \lambda_k^{[N]}\big)^n
=(e_b^\xi)^\top \big(T^{[N]}\big)^ne^\nu_a, & a&\in\{1,\dots, p\},& b&\in\{1,\dots,q\}.
\end{aligned}
\end{equation}
and the following degrees of precision
\[\begin{aligned}
 d_{b,a}(N)&=\left\lceil\frac{ N+2-a}{p}\right\rceil+\left\lceil\frac{ N+2-b}{q}\right\rceil-1,& a&\in\{1,\dots,p\}, & b&\in\{1,\dots,q\}.
\end{aligned} \]
 These degrees of precision or $d_{b,a}(N)$, $a\in\{1,\dots,p\}$, $b\in\{1,\dots,q\}$, are the largest natural numbers such~that
\[\begin{aligned}
(e_b^\xi)^\top \big(T^{[N]}\big)^n e^\nu_a  &= \big(u_b^\xi\big)^\top T^n u^\nu_a, & 0&\leqslant n\leqslant d_{b,a}(N), & a&\in\{1,\dots,p\},& b&\in\{1,\dots,q\} .
\end{aligned}\]

\section{The Favard spectral theorem for unbounded banded matrices}

In \cite{BFM1,BFM4} we proved the following result. A triangular matrix is said to be \emph{totally positive} if all its nontrivial minors are strictly positive; see~\cite{BFM4}.

\begin{teo}[Favard spectral representation for bounded banded matrices with PBF]
 Let us assume that the matrix $T$ is bounded and admits a positive bidiagonal factorization. Then the following versions of the Favard spectral theorem hold:
\begin{enumerate}[\rm i)]
  \item Let the initial condition matrices $A_0$ and $B_0$ be determined by the positive bidiagonal factorization up to left multiplication by an upper unitriangular positive matrix and a lower unitriangular positive matrix, respectively, as described in \cite{BFM1}. Then there exist $p \times q$ nondecreasing positive functions $\psi_{b,a}$, $a\in\{1,\dots,p\}$ and $b\in\{1,\dots,q\}$, and corresponding positive Lebesgue--Stieltjes measures $\d\psi_{b,a}$ such that the following mixed-type multiple biorthogonality relations hold:
\[  \begin{aligned}
   \sum_{a=1}^p\sum_{b=1}^q \int_{\Delta} B^{(b)}_l(x)\,\d\psi_{b,a}(x)\, A^{(a)}_{k}(x)
   &= \delta_{k,l}, & k,l&\in\N_0.
  \end{aligned}\]
  Moreover, the degrees of these mixed-type multiple orthogonal polynomials satisfy
\begin{align*}
  \begin{aligned}
   \deg A_n^{(a)}&\leq\left\lceil\frac{n+2-a}{p}\right\rceil-1,
   & \deg B_n^{(b)}&\leq\left\lceil\frac{n+2-b}{q}\right\rceil-1.
  \end{aligned}
\end{align*}
  Equality is attained at least when $n=p+a-1$ for $a\in\{1,\ldots,p\}$ (left polynomials), and when $n=Nq+b-1$ for $b\in\{1,\ldots,q\}$ (right polynomials).
  
  \item Let the initial condition matrices $A_0$ and $B_0$ be determined by the positive bidiagonal factorization up to left multiplication by an upper unitriangular totally positive matrix  and a lower unitriangular totally positive matrix, respectively, as described in \cite{BFM4}. Then there exist $p \times q$ nondecreasing positive functions $\psi_{b,a}$, $a\in\{1,\dots,p\}$ and $b\in\{1,\dots,q\}$, and corresponding positive Lebesgue--Stieltjes measures $\d\psi_{b,a}$ such that the mixed-type multiple biorthogonality relations hold:
\[  \begin{aligned}
   \sum_{a=1}^p\sum_{b=1}^q \int_{\Delta} B^{(b)}_l(x)\,\d\psi_{b,a}(x)\, A^{(a)}_{k}(x)
   &= \delta_{k,l}, & k,l&\in\N_0.
  \end{aligned}\]
  In this case, the degrees are maximal:
  \[
  \begin{aligned}
   \deg A_n^{(a)}&=\left\lceil\frac{n+2-a}{p}\right\rceil-1,
   & \deg B_n^{(b)}&=\left\lceil\frac{n+2-b}{q}\right\rceil-1.
  \end{aligned}
  \]
 \end{enumerate}
\end{teo}

Our aim here is to extend this result to unbounded matrices. More precisely, we prove the following theorem.

\begin{teo}[Favard spectral representation for unbounded banded matrices with PBF]\label{teo:Favard-unbounded-shifted}
	Assume that the matrix $T$ is unbounded and that there exists a sequence $(s_N)_{N\ge 0}$ with $s_N\ge 0$ such that, for every $N$, the truncation
	\[
	T^{[N]}+s_N I_{N+1}
	\]
	admits a positive bidiagonal factorization.
	Let the initial condition matrices $A_0$ and $B_0$ be determined by this factorization up to left multiplication by an upper unitriangular positive matrix and a lower unitriangular positive matrix, respectively, as described in \cite{BFM1}.
	
	Then there exist $p\times q$ nondecreasing positive functions $\psi_{b,a}$, $a\in\{1,\dots,p\}$ and $b\in\{1,\dots,q\}$, and the corresponding positive Lebesgue--Stieltjes measures $\d\psi_{b,a}$, such that
	\[
	\begin{aligned}
		\sum_{a=1}^p\sum_{b=1}^q \int_{\mathbb R} B^{(b)}_l(x)\,x^n\,\d\psi_{b,a}(x)\, A^{(a)}_{k}(x)
		&= (T^{n})_{k,l}, & k,l&\in\N_0,\ n\in\N_0.
	\end{aligned}
	\]
	In particular, the mixed-type multiple biorthogonality relations hold:
	\[
	\begin{aligned}
		\sum_{a=1}^p\sum_{b=1}^q \int_{\mathbb R} B^{(b)}_l(x)\,\d\psi_{b,a}(x)\, A^{(a)}_{k}(x)
		&= \delta_{k,l}, & k,l&\in\N_0,
	\end{aligned}
	\]
	as well as the mixed-type orthogonality relations:
	\[
	\begin{aligned}
		\sum_{a=1}^p\int_{\mathbb R} x^n\,\d\psi_{b,a}(x)\,A^{(a)}_{m}(x)&=0, &
		n&\in\left\{0,\dots,\left \lceil \frac{m+1-b}{q}\right \rceil -1\right\},& b&\in\{1,\dots,q\},\\
		\sum_{b=1}^q\int_{\mathbb R} B^{(b)}_{m}(x)\,\d\psi_{b,a}(x)\,x^n&=0, &
		n&\in\left\{0,\dots,\left \lceil \frac{m+1-a}{p}\right \rceil -1\right\}, & a&\in\{1,\dots,p\}.
	\end{aligned}
	\]
	
	The degrees satisfy
	\[
	\begin{aligned}
		\deg A_n^{(a)}&\le \left\lceil\frac{n+2-a}{p}\right\rceil-1,
		& \deg B_n^{(b)}&\le \left\lceil\frac{n+2-b}{q}\right\rceil-1,
	\end{aligned}
	\]
	with equality attained at least when $n=\left\lceil\frac{q}{p}\right\rceil p+a-1$ for $a\in\{1,\ldots,p\}$, and when $n=\left\lceil\frac{p}{q}\right\rceil q+b-1$ for $b\in\{1,\ldots,q\}$.
	Moreover, if the initial condition matrices $A_0$ and $B_0$ are determined by the positive bidiagonal factorization up to left multiplication by an upper unitriangular totally positive matrix and a lower unitriangular totally positive matrix, respectively, then equality holds in the above degree bounds for all $n\in\N_0$ and for all $a\in\{1,\ldots,p\}$ and $b\in\{1,\ldots,q\}$.
\end{teo}
For the proof we base our analysis on \cite{Chihara}. In particular, we recall the following results due to~Helly.

\begin{enumerate}[\rm i)]
 \item \textbf{Selection principle.}
 
{\noindent}\emph{Let $\{\psi_n\}$ be a uniformly bounded sequence of nondecreasing functions defined on $(-\infty,\infty)$. Then $\{\psi_n\}$ admits a subsequence which converges on $(-\infty,\infty)$ to a bounded nondecreasing function.}
 
 \item \textbf{Helly's second theorem.}
 
{\noindent}\emph{Let $\{\psi_n\}$ be a uniformly bounded sequence of nondecreasing functions defined on a compact interval $[\alpha,\beta]$, and suppose that it converges on $[\alpha,\beta]$ to a limit function $\psi$. Then, for every real-valued function $f$ continuous on~$[\alpha,\beta]$,
  \[
  \lim_{n\to\infty}\int_\alpha^\beta f\,\d\psi_n=\int_\alpha^\beta f\,\d\psi.
  \]
 }
\end{enumerate}

We also need the following lemmas:

\begin{lemma}\label{lem:banded-stabilization}
	Let $T$ be a semi-infinite matrix with band structure
	\[
	\begin{aligned}
		T_{i,j}&=0, & j&<i-p \ \text{or}\ j>i+q,
	\end{aligned}
	\]
	for some $p,q\in\N$. Let $T^{[N]}$ be the principal truncation of size $N+1$.
	Let $\nu\in\mathbb{C}^{p\times p}$ and $\xi\in\mathbb{C}^{q\times q}$ be invertible lower triangular matrices.
	Fix $a\in\{1,\dots,p\}$ and $b\in\{1,\dots,q\}$, and consider the vectors (Definition~9.3)
	\begin{equation}\label{eq:ua-ub-def}
		\begin{aligned}
			u_a^\nu&\coloneq E_{[p]}^\top \nu^{-\top} e_a^{[p]}, &
			u_b^\xi&\coloneq (e_b^{[q]})^\top \xi^{-1} E_{[q]}, &
			E_{[r]}&\coloneq
			\left[\begin{NiceArray}{cccc|cccccc}
				1&0&\Cdots&0 & 0&\Cdots&&& &\\
				0&\Ddots^{\text{$r$ times}}&\Ddots& \Vdots &\Vdots^{\text{$r$ times}}&&&&\\
				\Vdots&\Ddots&\Ddots& 0 &&&&&\\
				0&\Cdots&0& 1 &0&\Cdots&&&
			\end{NiceArray}\right].
		\end{aligned}
	\end{equation}
	Then, for every $n\in\N_0$ and every $N\in\N_0$ such that
	\begin{align}\label{eq:dba-def}
		n&\le d_{b,a}(N)\coloneq
		\left\lceil\frac{N+2-a}{p}\right\rceil+\left\lceil\frac{N+2-b}{q}\right\rceil-1,
	\end{align}
	we have the exact stabilization identity
	\begin{align}\label{eq:stab-identity}
		u_b^\xi\,T^n\,u_a^\nu
		&=(e_b^\xi)^\top \bigl(T^{[N]}\bigr)^n e_a^\nu.
	\end{align}
\end{lemma}

\begin{proof}
	Write the path expansion
	\begin{align}\label{eq:path-expansion}
		(T^n)_{i,j}
		&=\sum_{i_1,\dots,i_{n-1}\in\N_0}
		T_{i,i_1}\,T_{i_1,i_2}\cdots T_{i_{n-1},j}.
	\end{align}
	A term in \eqref{eq:path-expansion} can be nonzero only if each factor respects the band, namely
	\[
	\begin{aligned}
		i_{k+1}&\in\{i_k-p,\dots,i_k+q\}, & k&\in\{0,\dots,n-1\},
	\end{aligned}
	\]
	with $i_0=i$ and $i_n=j$. Hence, along any contributing path, the index can increase by at most $q$
	per step and decrease by at most $p$ per step.
	
	Assume that a contributing path starting at some $i\in\{0,\dots,b-1\}$ and ending at some
	$j\in\{0,\dots,a-1\}$ visits an index $>N$. Let $N+1$ be the first index strictly larger than $N$
	visited by the path, and let $s$ be the first time such that $i_s=N+1$. Since each step increases
	the index by at most $q$, reaching $N+1$ from $i_0\le b-1$ requires
	\[
	s\ge \left\lceil\frac{N+1-(b-1)}{q}\right\rceil
	=\left\lceil\frac{N+2-b}{q}\right\rceil.
	\]
	Similarly, since each step decreases the index by at most $p$, going from $N+1$ down to $j\le a-1$
	requires
	\[
	n-s\ge \left\lceil\frac{N+1-(a-1)}{p}\right\rceil
	=\left\lceil\frac{N+2-a}{p}\right\rceil.
	\]
	Therefore any contributing path that leaves $\{0,1,\dots,N\}$ must have length
	\[
	n\ge
	\left\lceil\frac{N+2-b}{q}\right\rceil
	+
	\left\lceil\frac{N+2-a}{p}\right\rceil,
	\]
	which contradicts \eqref{eq:dba-def}. Hence, if \eqref{eq:dba-def} holds, every nonzero
	contributing path remains within $\{0,1,\dots,N\}$, and consequently
	\[
	(T^n)_{i,j}=\bigl((T^{[N]})^n\bigr)_{i,j},
	\]
	for all $i\in\{0,\dots,b-1\}$ and $j\in\{0,\dots,a-1\}$.
	
	Since $\nu$ and $\xi$ are lower triangular, $\nu^{-\top}$ is upper triangular and $\xi^{-1}$ is lower triangular.
	Thus, by \eqref{eq:ua-ub-def}, the vector $u_a^\nu$ is supported in $\{0,\dots,a-1\}$ and the row vector
	$u_b^\xi$ is supported in $\{0,\dots,b-1\}$. Multiplying the above entrywise identity by the corresponding
	coefficients yields \eqref{eq:stab-identity}.
\end{proof}

\begin{lemma}\label{lem:inverse-bound}
	\begin{enumerate}
		\item Let us use the notation
		\[
		{\mathscr N}_{p,q,a,b}(n)\coloneq
		\frac{n+1-\bigl(\frac{2-a}{p}+\frac{2-b}{q}\bigr)}{\frac1p+\frac1q}.
		\]
		Then, for every $N\in\N_0$ such that $N\ge \mathscr N_{p,q,a,b}(n)$, we have
		\begin{align}\label{eq:inv-bound-conclusion}
			n&\le d_{b,a}(N).
		\end{align}
		\item In particular, defining the moment $m_{n,a,b} \coloneq u_b^\xi\,T^n\,u_a^\nu$, Lemma~\ref{lem:banded-stabilization}
		yields
		\[
		\begin{aligned}
			(e_b^\xi)^\top \bigl(T^{[N]}\bigr)^n e_a^\nu&=m_{n,a,b},
			& \text{for all } N&\ge \mathscr N_{p,q,a,b}(n).
		\end{aligned}
		\]
		Consequently, the quantity $(e_b^\xi)^\top \bigl(T^{[N]}\bigr)^n e_a^\nu$ stabilizes for
		$N\ge \mathscr N_{p,q,a,b}(n)$.
	\end{enumerate}
\end{lemma}

\begin{proof}
	Using the inequality $\lceil x\rceil\ge x$, we obtain
	\[
	d_{b,a}(N)
	=\left\lceil\frac{N+2-a}{p}\right\rceil+\left\lceil\frac{N+2-b}{q}\right\rceil-1
	\ge
	\frac{N+2-a}{p}+\frac{N+2-b}{q}-1.
	\]
	Thus, the condition $n\le d_{b,a}(N)$ is ensured as soon as
	\[
	n\le
	\frac{N+2-a}{p}+\frac{N+2-b}{q}-1
	=
	N\Bigl(\frac1p+\frac1q\Bigr)+\Bigl(\frac{2-a}{p}+\frac{2-b}{q}-1\Bigr).
	\]
	Rearranging this inequality gives $N\ge {\mathscr N}_{p,q,a,b}(n)$, which proves \eqref{eq:inv-bound-conclusion}.
	The stabilization statement in (ii) then follows from Lemma~\ref{lem:banded-stabilization}.
\end{proof}

\begin{proof}[Proof of the Favard spectral theorem]
	\noindent\textbf{Step 0 (assumption and shifted step functions).}
	Let us assume that for every $N\in\N_0$ there exists a shift $s_N\ge 0$ such that the shifted truncation
	\[
	A_N\coloneq T^{[N]}+s_N I_{N+1}
	\]
	admits a positive bidiagonal factorization (PBF). Fix $(a,b)$. Let $\d\psi^{[N],s_N}_{b,a}$ be the
	corresponding discrete Gauss-type quadrature measure attached to $A_N$, and let
	$\psi^{[N],s_N}_{b,a}$ be its right-continuous distribution (step) function, i.e.
	\[
	\psi^{[N],s_N}_{b,a}(x)\coloneq \d\psi^{[N],s_N}_{b,a}\bigl((-\infty,x]\bigr).
	\]
	Since the shift depends on $N$, we recenter the distributions by translation and define
	\begin{equation}\label{eq:shifted-distribution-function}
		\begin{aligned}
			\psi^{[N]}_{b,a}(x)&\coloneq \psi^{[N],s_N}_{b,a}(x+s_N),&& x\in\R.
		\end{aligned}
	\end{equation}
	Equivalently, in terms of measures,
	\begin{equation}\label{eq:shifted-measure-pushforward}
		\d\psi^{[N]}_{b,a}(E)=\d\psi^{[N],s_N}_{b,a}(E+s_N)
		\quad\text{for every Borel set }E\subset\R,
	\end{equation}
	i.e.\ $\d\psi^{[N]}_{b,a}$ is the push-forward of $\d\psi^{[N],s_N}_{b,a}$ under the map $x\mapsto x-s_N$.
	
	More explicitly, if $0\le\lambda^{[N]}_{N+1}<\lambda^{[N]}_{N}<\cdots<\lambda^{[N]}_1$ are the ordered simple eigenvalues of $A_N$, and $\rho^{[N]}_{n,b}$, $\mu^{[N]}_{n,a}$ are the Christoffel numbers (which are positive due to the PBF assumption on $A_N$), we have 	
	\[
	\psi^{[N]}_{b,a}(x)=
	\begin{cases}
		0, & x<\lambda^{[N]}_{N+1}-s_N,\\[2pt]
		\rho^{[N]}_{1,b}\mu^{[N]}_{1,a}+\cdots+\rho^{[N]}_{k,b}\mu^{[N]}_{k,a}, & \lambda^{[N]}_{k+1}-s_N\leqslant x< \lambda^{[N]}_{k}-s_N, \quad k\in\{0,\dots,N-1\},\\[2pt]
		\rho^{[N]}_{1,b}\mu^{[N]}_{1,a}+\cdots+\rho^{[N]}_{N+1,b}\mu^{[N]}_{N+1,a} ,
		& x \geqslant \lambda^{[N]}_{1}-s_N,
	\end{cases}
	\]
	and the corresponding translated measures are
	\[
	\d\psi^{[N]}_{b,a}
	=\sum_{k=1}^{N+1}\rho^{[N]}_{k,b}\mu^{[N]}_{k,a}\,
	\delta\!\left(x-(\lambda^{[N]}_{k}-s_N)\right).
	\]
	
	\medskip
	
	\noindent\textbf{Step 1 (Helly compactness on $[\alpha,\beta]$).}
	Since translations preserve monotonicity and total mass, for every $a\in\{1,\dots,p\}$ and $b\in\{1,\dots,q\}$ the functions $\psi^{[N]}_{b,a}$ are nonnegative, nondecreasing, right-continuous, and uniformly bounded by the total mass $(\xi^{-1}I_{q,p}\nu^{-\top})_{b,a}$.
	
	Hence $\left\{\psi^{[N]}_{b,a}\right\}_{N\in\N_0}$ is a uniformly bounded family of nondecreasing, right-continuous
	functions. By Helly's selection principle there exists a subsequence
	$\left\{\psi^{[N_i]}_{b,a}\right\}_{i\ge1}$ converging pointwise at all continuity points to a bounded
	nondecreasing function $\psi_{b,a}$. Moreover, by Helly's second theorem, for every compact interval
	$[\alpha,\beta]$ and every continuous function $f$ on $[\alpha,\beta]$,
	\begin{equation}\label{eq:helly-second-general}
		\lim_{i\to\infty}\int_\alpha^\beta f(x)\,\d\psi^{[N_i]}_{b,a}(x)
		=
		\int_\alpha^\beta f(x)\,\d\psi_{b,a}(x).
	\end{equation}
	
	\medskip
	
	\noindent\textbf{Step 2 (quadrature for $(x-s_N)^n$ and recentered moments).}
Let us recall that for  $N,k\in\N_0$ we have
	\begin{align}\label{eq:quad-monomials-step2}
		\int_{0}^{\infty} x^k\,\d\psi^{[N],s_N}_{b,a}(x)
		&=(e_b^\xi)^\top (A_N)^k e_a^\nu.
	\end{align}
	We choose $N\in\N_0$ such that
	\begin{equation}\label{eq:Nchoice-step2}
		N\ge \left\lceil \mathscr N_{p,q,a,b}(n)\right\rceil,
	\end{equation}
	so that, by Lemma~\ref{lem:inverse-bound}, we have $n\le d_{b,a}(N)$, and therefore
	\eqref{eq:quad-monomials-step2} holds for all $k\in\{0,1,\dots,n\}$.
	
	Since
	\[
	(x-s_N)^n=\sum_{k=0}^n \binom{n}{k}(-s_N)^{\,n-k}x^k,
	\]
	linearity of the integral and \eqref{eq:quad-monomials-step2} yield
	\begin{align}\label{eq:quad-shifted-poly-step2}
		\int_{0}^{\infty} (x-s_N)^n\,\d\psi^{[N],s_N}_{b,a}(x)
		&=
		\sum_{k=0}^n \binom{n}{k}(-s_N)^{\,n-k}
		(e_b^\xi)^\top (A_N)^k e_a^\nu \notag\\
		&=
		(e_b^\xi)^\top (A_N-s_N I_{N+1})^n e_a^\nu.
	\end{align}
	Moreover,
	\[
	(A_N-s_N I_{N+1})^n
	=\Bigl(T^{[N]}+s_N I_{N+1}-s_N I_{N+1}\Bigr)^n
	=\bigl(T^{[N]}\bigr)^n,
	\]
so that
	\begin{align}\label{eq:AN-to-TN-step2}
		\int_{0}^{\infty} (x-s_N)^n\,\d\psi^{[N],s_N}_{b,a}(x)
		&=
		(e_b^\xi)^\top \bigl(T^{[N]}\bigr)^n e_a^\nu.
	\end{align}
	
	Now set
	\[
	m_{n,a,b}\coloneq (u_b^\xi)\,T^n\,u_a^\nu.
	\]
	Since $n\le d_{b,a}(N)$, Lemma~\ref{lem:banded-stabilization} gives
	\begin{align}\label{eq:stabilization-step2}
		(e_b^\xi)^\top \bigl(T^{[N]}\bigr)^n e_a^\nu&=m_{n,a,b}.
	\end{align}
	Combining \eqref{eq:AN-to-TN-step2} and \eqref{eq:stabilization-step2}, we conclude that
\begin{equation}
		\begin{aligned}
		\int_{0}^{\infty} (x-s_N)^n\,\d\psi^{[N],s_N}_{b,a}(x)&=m_{n,a,b},
		&& N\ge \left\lceil \mathscr N_{p,q,a,b}(n)\right\rceil.
	\end{aligned}\label{eq:shifted-moment-equals-m-step2}
\end{equation}
	
	Finally, by the push-forward relation \eqref{eq:shifted-measure-pushforward}, for any polynomial $p$
	we have
	\[
	\int_{\R} p(u)\,\d\psi^{[N]}_{b,a}(u)
	=
	\int_{0}^{\infty} p(x-s_N)\,\d\psi^{[N],s_N}_{b,a}(x),
	\]
	and in particular, taking $p(u)=u^n$,
\[	\begin{aligned}
		\int_{\R} u^n\,\d\psi^{[N]}_{b,a}(u)
		&=
		\int_{0}^{\infty} (x-s_N)^n\,\d\psi^{[N],s_N}_{b,a}(x)
		=
		m_{n,a,b},
		& N&\ge \left\lceil \mathscr N_{p,q,a,b}(n)\right\rceil.
	\end{aligned}\label{eq:recentered-moment-step2}\]
	
	\noindent\textbf{Step 3 (tails and passage to the limit).}
	Fix $n\in\N_0$ and $\alpha<\beta$, and set $R\coloneq \max\{|\alpha|,|\beta|\}$. Let
	$\{N_i\}_{i\ge1}$ be the subsequence from Step~1, and take $i$ large enough so that
	$2n+2\le d_{b,a}(N_i)$ (equivalently, $N_i\ge \left\lceil \mathscr N_{p,q,a,b}(2n+2)\right\rceil$). Then Step~2 yields
	\begin{equation}\label{eq:moments-Ni-step3}
\begin{aligned}
			\int_{\R} x^n\,\d\psi^{[N_i]}_{b,a}(x)&=m_{n,a,b},
		&
		\int_{\R} x^{2n+2}\,\d\psi^{[N_i]}_{b,a}(x)&=m_{2n+2,a,b}.
\end{aligned}
	\end{equation}
	We write
	\begin{align}\label{eq:decomp-step3}
		m_{n,a,b}-\int_{\alpha}^{\beta} x^n\,\d\psi_{b,a}(x)
		&=
		\Bigl(\int_{\R} x^n\,\d\psi^{[N_i]}_{b,a}(x)-\int_{\alpha}^{\beta} x^n\,\d\psi^{[N_i]}_{b,a}(x)\Bigr)
		+
		\Bigl(\int_{\alpha}^{\beta} x^n\,\d\psi^{[N_i]}_{b,a}(x)-\int_{\alpha}^{\beta} x^n\,\d\psi_{b,a}(x)\Bigr).
	\end{align}
	
	\smallskip
	\noindent\emph{Convergence on $[\alpha,\beta]$.}
	By Helly's second theorem \eqref{eq:helly-second-general}, applied to the continuous function $f(x)=x^n$ on $[\alpha,\beta]$,
	the second term in the RHS \eqref{eq:decomp-step3} tends to $0$ as $i\to\infty$.
	
	\smallskip
	\noindent\emph{Tail estimate.}
	For the first term in \eqref{eq:decomp-step3} we note that
	$\R\setminus[\alpha,\beta]\subset\{|x|>R\}$. Since $\d\psi^{[N_i]}_{b,a}$ is a nonnegative measure,
	\begin{align}\label{eq:tail-1-step3}
		\Bigl|\int_{\R\setminus[\alpha,\beta]} x^n\,\d\psi^{[N_i]}_{b,a}(x)\Bigr|
		&\le
		\int_{\R\setminus[\alpha,\beta]} |x|^n\,\d\psi^{[N_i]}_{b,a}(x)
		\le
		\int_{|x|>R} |x|^n\,\d\psi^{[N_i]}_{b,a}(x).
	\end{align}
	Moreover, on $\{|x|>R\}$ we have $|x|\ge R$, hence $|x|^{-(n+2)}\le R^{-(n+2)}$, and therefore
	\begin{equation}\label{eq:pointwise-tail-step3}
		|x|^n
		=
		|x|^{2n+2}\,|x|^{-(n+2)}
		\le
		R^{-(n+2)}\,|x|^{2n+2},
		\qquad |x|>R.
	\end{equation}
	Integrating the pointwise bound \eqref{eq:pointwise-tail-step3} against the nonnegative measure
	$\d\psi^{[N_i]}_{b,a}$ over $\{|x|>R\}$ yields
	\begin{align}\label{eq:tail-2-step3}
		\int_{|x|>R} |x|^n\,\d\psi^{[N_i]}_{b,a}(x)
		&\le
		R^{-(n+2)}\int_{|x|>R} |x|^{2n+2}\,\d\psi^{[N_i]}_{b,a}(x).
	\end{align}
	Since $|x|^{2n+2}\ge 0$ and $\{|x|>R\}\subset\R$, we also have
	\begin{align}\label{eq:tail-3-step3}
		\int_{|x|>R} |x|^{2n+2}\,\d\psi^{[N_i]}_{b,a}(x)
		&\le
		\int_{\R} |x|^{2n+2}\,\d\psi^{[N_i]}_{b,a}(x).
	\end{align}
	Combining \eqref{eq:tail-1-step3}--\eqref{eq:tail-3-step3} we obtain
	\begin{equation}\label{eq:tail-final-step3}
		\Bigl|\int_{\R\setminus[\alpha,\beta]} x^n\,\d\psi^{[N_i]}_{b,a}(x)\Bigr|
		\le
		R^{-(n+2)}\int_{\R} |x|^{2n+2}\,\d\psi^{[N_i]}_{b,a}(x).
	\end{equation}
	Since $2n+2$ is even, $|x|^{2n+2}=x^{2n+2}$, and by \eqref{eq:moments-Ni-step3},
	\[
	\int_{\R} |x|^{2n+2}\,\d\psi^{[N_i]}_{b,a}(x)
	=
	\int_{\R} x^{2n+2}\,\d\psi^{[N_i]}_{b,a}(x)
	=
	m_{2n+2,a,b}.
	\]
	Therefore,
	\begin{equation}\label{eq:tail-moment-step3}
		\Bigl|\int_{\R\setminus[\alpha,\beta]} x^n\,\d\psi^{[N_i]}_{b,a}(x)\Bigr|
		\le
		R^{-(n+2)}\,m_{2n+2,a,b}.
	\end{equation}
	
	\smallskip
	\noindent\emph{Conclusion.}
	Letting $i\to\infty$ in \eqref{eq:decomp-step3} and using the convergence on $[\alpha,\beta]$ and
	the tail bound \eqref{eq:tail-moment-step3}, we obtain
	\[
	\Bigl|m_{n,a,b}-\int_{\alpha}^{\beta} x^n\,\d\psi_{b,a}(x)\Bigr|
	\le
	R^{-(n+2)}\,m_{2n+2,a,b}.
	\]
	Finally, letting $\alpha\to-\infty$ and $\beta\to+\infty$ (so that $R\to+\infty$) yields
	\[
	\int_{\R} x^n\,\d\psi_{b,a}(x)=m_{n,a,b}=(u_b^\xi)\,T^n\,u_a^\nu,
	\qquad n\in\N_0.
	\]
\end{proof}

In both cases, allowing positive or totally positive triangular freedom in the initial conditions of the mixed-type orthogonal polynomials, the Gauss quadrature formulas derived from \eqref{eq:moments_discrete}
and the above proof~read
\begin{align*}
\begin{aligned}
 \int_{0}^\infty x^n\,\d\psi_{b,a}(x)
 = \sum_{k=1}^{N+1} \rho_{k,b}^{[N]}\mu_{k,a}^{[N]}\bigl(\lambda_k^{[N]}\bigr)^n,
 \qquad
 0\leqslant n\leqslant d_{b,a}(N),
\end{aligned}
\end{align*}
for $a\in\{1,\ldots,p\}$ and $b\in\{1,\ldots,q\}$. Here the degrees of precision $d_{b,a}(N)$ are optimal: for any larger power, a positive remainder appears and exactness is lost.

\subsection{Case study: unbounded Jacobi matrices and shifted PBF for truncations}\label{subsec:case-study-jacobi-unbounded}

A particularly transparent instance of the ``shifted--PBF truncation'' hypothesis in
Theorem~\ref{teo:Favard-unbounded-shifted} arises in the context of unbounded Jacobi matrices.
Let
\[
\begin{aligned}
	J&=\begin{bNiceArray}{ccccc}[margin]
	b_0 & c_0&0 &\Cdots[shorten-end=-15pt] & \phantom{a}\\
	a_1 & b_1 & c_1 &\Ddots[shorten-end=-15pt]&\phantom{a}\\
	0& a_2 & b_2 & c_2 &  \phantom{a}\\
	\Vdots[shorten-end=-5pt]&\Ddots[shorten-end=-10pt]& \Ddots[shorten-end=-10pt] & \Ddots[shorten-end=-10pt]&\Ddots[shorten-end=-5pt]\\
	\phantom{a}&	\phantom{a}&	\phantom{a}&	\phantom{a}&	\phantom{a}
\end{bNiceArray},
&& a_n>0,\ c_n>0,\ n\ge 0,
\end{aligned}
\]
where the diagonal $(b_n)_{n\ge 0}$ is real and may be unbounded (consequently, $J$ may be unbounded as an operator).
For each $N\in\N_0$, denote by $J^{[N]}$ the principal truncation of size $N+1$.

The goal of this case study is twofold:
(i) to provide an explicit, entrywise choice of shifts $s_N\ge 0$ such that each shifted truncation $J^{[N]}+s_N I_{N+1}$
admits a PBF, and
(ii) to show how the Favard spectral representation for the unbounded case follows immediately from the
general apparatus developed above.

\begin{lemma}[A convenient lower spectral bound for a truncation]\label{lem:jacobi-lower-bound}
	Let $M\in\R^{(N+1)\times(N+1)}$ be any real matrix and let $\|\cdot\|_\infty$ denote the induced $\ell^\infty$ operator norm, i.e.,
	\[
	\|M\|_\infty=\max_{0\le i\le N}\sum_{j=0}^N |M_{ij}|.
	\]
	Then every eigenvalue $\lambda$ of $M$ satisfies $|\lambda|\le \|M\|_\infty$. In particular,
	\[
	\sigma(M)\subset \{z\in\C:\ |z|\le \|M\|_\infty\}.
	\]
\end{lemma}

\begin{proof}
	If $\lambda$ is an eigenvalue of $M$, then $\lambda$ is an eigenvalue of the bounded linear operator
	$x\mapsto Mx$ on the Banach space $(\C^{N+1},\|\cdot\|_\infty)$, which implies $|\lambda|\le \|M\|_\infty$.
\end{proof}

\begin{teo}[Shifted PBF for Jacobi truncations]\label{teo:jacobi-shifted-pbf}
	Let $J$ be a (possibly unbounded) Jacobi matrix with $a_n>0$ and $c_n>0$.
	For the truncation $J^{[N]}$, we adopt the boundary convention $a_0=0$ and $c_N=0$ (since these entries do not appear in the truncated matrix).
	Define the shift
	\begin{equation}\label{eq:sN-infty-choice}
		s_N\coloneq \|J^{[N]}\|_\infty+1
		=\max_{0\le i\le N}\bigl(|a_i|+|b_i|+|c_i|\bigr)+1.
	\end{equation}
	Then the shifted truncation
	\[
	A_N\coloneq J^{[N]}+s_N I_{N+1}
	\]
	has all leading principal minors strictly positive and is an oscillatory matrix.
	Consequently, $A_N$ admits a positive bidiagonal factorization (PBF) for every $N\in\N_0$.
\end{teo}

\begin{proof}
	Fix $N\in\N_0$ and set $A_N=J^{[N]}+s_N I_{N+1}$ with $s_N$ defined as in \eqref{eq:sN-infty-choice}.
	For each $k\in\{0,1,\dots,N\}$, let $A_k$ denote the leading principal submatrix of $A_N$ of size $k+1$,
	given by
	\[
	A_k = J^{[k]}+s_N I_{k+1}.
	\]
	
	\smallskip
	\noindent\emph{Step 1: $-s_N$ lies strictly to the left of the spectrum of every $J^{[k]}$.}
	By Lemma~\ref{lem:jacobi-lower-bound},
	every eigenvalue $\lambda$ of the truncation $J^{[k]}$ satisfies $|\lambda|\le \|J^{[k]}\|_\infty$.
	Since $\|J^{[k]}\|_\infty\le \|J^{[N]}\|_\infty$ for all $k\le N$, we obtain the lower bound
	\[
	\lambda\ge -|\lambda| \ge -\|J^{[k]}\|_\infty \ge -\|J^{[N]}\|_\infty > -s_N.
	\]
	Thus, $-s_N$ is a strict lower bound for the spectrum of $J^{[k]}$ for every $k\le N$.
	
	\smallskip
	\noindent\emph{Step 2: Positivity of leading principal minors of $A_N$.}
	For each $k\le N$, let $P_{k+1}(x)\coloneq \det(x I_{k+1}-J^{[k]})$ be the characteristic polynomial of $J^{[k]}$.
	Then
	\[
	\det(A_k)=\det(J^{[k]}+s_N I_{k+1}) = \det\bigl((-s_N) I_{k+1}-J^{[k]}\bigr)
	= (-1)^{k+1} P_{k+1}(-s_N).
	\]
	Since $J^{[k]}$ is a symmetric tridiagonal matrix (modulo similarity) with non-vanishing off-diagonals, all its eigenvalues $\lambda_1, \dots, \lambda_{k+1}$ are real and simple. Therefore,
	\[
	P_{k+1}(-s_N) = \prod_{j=1}^{k+1} (-s_N - \lambda_j).
	\]
	By Step~1, $-s_N < \lambda_j$ for all $j$, so each factor $(-s_N - \lambda_j)$ is strictly negative.
	The product of $k+1$ negative numbers has sign $(-1)^{k+1}$.
	Consequently,
	\[
	\det(A_k)=(-1)^{k+1}\underbrace{P_{k+1}(-s_N)}_{(-1)^{k+1} |P(\dots)|} > 0,
	\]
	which proves that all leading principal minors of $A_N$ are strictly positive.
	
	\smallskip
	\noindent\emph{Step 3: Oscillatory implies PBF.}
	By construction, the subdiagonal and superdiagonal entries of $A_N$ coincide with those of $J^{[N]}$ and are strictly positive.
	Since $A_N$ is a Jacobi matrix with strictly positive leading principal minors (Step~2) and positive off-diagonals,
is a classical result that implies that $A_N$ is oscillatory, 	\cite[Chapter XIII,\S 9]{Gantmacher0} and \cite[Chapter 2,Theorem 11]{Gantmacher}.
	Finally, we recall the classical result that every oscillatory matrix admits an $LU$ factorization where the factors have strictly positive diagonals and, in the tridiagonal case, strictly positive off-diagonals.
	Thus, $A_N$ admits a PBF.
\end{proof}

\medskip

\noindent
\textbf{Connection with the unbounded Favard spectral theorem.}
Theorem~\ref{teo:jacobi-shifted-pbf} provides an explicit sequence $(s_N)_{N\ge 0}$ satisfying the hypothesis of
Theorem~\ref{teo:Favard-unbounded-shifted}: for each $N$, the shifted truncation $J^{[N]}+s_N I_{N+1}$ admits a PBF.
Therefore, the construction of the translated distribution functions and measures
\eqref{eq:shifted-distribution-function}--\eqref{eq:shifted-measure-pushforward},
combined with the moment identities from Step~2 and the Helly compactness argument in Step~3 of the main proof,
yields a matrix-valued measure $\d\Psi$ providing the Favard-type spectral representation for the unbounded Jacobi matrix $J$.

\medskip


\section{Conclusions and outlook}

In this work we extend to the unbounded setting the spectral Favard theorem for banded matrices admitting a positive bidiagonal factorization (PBF). The bounded case in~\cite{BFM1} already relies on a compactness step for the finite-level spectral data, implemented via Helly-type selection theorems (see~\cite{Chihara}), to pass from truncations to a limiting matrix of measures. The genuinely new feature in the unbounded regime is that one must allow an $N$-dependent shift: for each truncation one considers $T^{[N]}+s_N I_{N+1}$, with $s_N>0$ chosen so that a PBF is available, and one then proves that the associated measures remain uniformly controlled as $N\to\infty$. This requires a finer analysis than in the bounded case, aimed in particular at controlling the tails and preventing loss of mass at infinity.

More precisely, we work with the shifted truncations $T^{[N]}+s_N I_{N+1}$ admitting a strict PBF and with the associated mixed-type multiple biorthogonal polynomials. At the finite level, Gaussian quadrature expresses the relevant moments through positive Christoffel numbers. In the unbounded case, these identities are combined with tail estimates to obtain uniform tightness and variation bounds for the corresponding distribution functions. This, in turn, makes it possible to apply Helly-type compactness principles (as presented in~\cite{Chihara}) to extract convergent subsequences. Identifying the limits and verifying that they reproduce the prescribed moments yields a matrix of positive Lebesgue--Stieltjes measures realizing the mixed-type multiple biorthogonality relations and the spectral representation of the powers $T^n$ in the unbounded case.

The argument preserves the constructive spirit of~\cite{BFM1} while addressing the analytic difficulties specific to the unbounded setting: controlling the limiting procedure and ensuring that positivity is not lost. As in the bounded case, the choice of initial conditions is decisive. Positivity of the triangular normalizations guarantees admissible degree bounds, whereas total positivity enforces normality and ensures that the degree upper bounds are attained for all colors and all indices.

Several directions remain open. From a structural viewpoint, it is natural to ask to what extent the method extends to broader classes of banded operators admitting factorizations beyond the bidiagonal case, or to settings where positivity holds only partially. Another problem concerns the fine spectral features of the resulting measures in the unbounded regime, including growth, regularity, and determinacy of the associated moment problems. Finally, the interplay between total positivity, factorization theory, and mixed-type multiple orthogonality suggests further connections with classical and matrix-valued extensions of orthogonal polynomial theory. A particularly promising direction is the application to continuous-time Markov processes with unbounded generator matrices, where the spectral representation may provide new tools to study recurrence, ergodicity, and stationary behavior; for related results in the discrete-time setting, see~\cite{JP,BFM6}.

\section*{Declarations}

\subsection*{Competing interests}
The authors declare that they have no competing interests.

\subsection*{Author contributions}
All authors contributed to the conception of the work, the development of the theoretical results, and the writing of the manuscript. All authors read and approved the final manuscript.

\subsection*{Use of artificial intelligence}
Artificial intelligence tools were used solely to improve the presentation (e.g.\ language and readability). All AI-assisted edits were reviewed and validated by the authors, who remain fully responsible for the content of the manuscript.

\subsection*{Data availability}
No datasets were generated or analysed during the current study.

\subsection*{Code availability}
No code was used or generated for the current study.

\subsection*{Ethics approval}
Not applicable.

\subsection*{Consent to participate}
Not applicable.

\subsection*{Consent for publication}
Not applicable.

\section*{Acknowledgments}

AB was financially supported by the Funda\c c\~ao para a Ci\^encia e a Tecnologia (Portuguese Foundation for Science and Technology) under the scope of the projects
\hyperref{https://doi.org/10.54499/UID/00324/2025}{}{}{\texttt{UID/00324/2025}}
(Centre for Mathematics of the University of Coimbra).

AF acknowledges CIDMA Center for Research and Development in Mathematics and Applications (University of Aveiro) and the Portuguese Foundation for Science and Technology (FCT) within project 
\hyperref{https://doi.org/10.54499/UID/04106/2025}{}{}{\texttt{UID/04106/2025}}.

MM acknowledges Spanish ``Agencia Estatal de Investigación'' research projects [PID2021- 122154NB-I00], \emph{Ortogonalidad y Aproximación con Aplicaciones en Machine Learning y Teoría de la Probabilidad} and [PID2024-155133NB-I00], \emph{Ortogonalidad, aproximación e integrabilidad: aplicaciones en procesos estocásticos clásicos y cuánticos}.

\printbibliography

@article{afm,
	author  = {\'Alvarez-Fern\'andez, Carlos and Fidalgo, Ulises and Ma{\~n}as, Manuel},
	title   = {Multiple orthogonal polynomials of mixed type: {G}auss--{B}orel factorization and the multi-component {2D} {T}oda hierarchy},
	journal = {Advances in Mathematics},
	volume  = {227},
	year    = {2011},
	pages   = {1451--1525},
	doi     = {10.1016/j.aim.2011.03.008} 
}

@article{JP,
	author  = {Branquinho, Am\'{\i}lcar and D{\'\i}az, Juan E. F. and Foulqui\'e-Moreno, Ana and Ma{\~n}as, Manuel and \'Alvarez-Fern\'andez, Carlos},
	title   = {Jacobi-Pi{\~n}eiro random walks},
	journal = {Revista de la Real Academia de Ciencias Exactas, F{\'\i}sicas y Naturales. Serie A. Matem{\'a}ticas},
	volume  = {118},
	year    = {2024},
	doi     = {10.1007/s13398-023-01510-x},
	eid     = {15}

}

@article{BFM1,
	author  = {Branquinho, Am\'{\i}lcar and Foulqui\'e-Moreno, Ana and Ma{\~n}as, Manuel},
	title   = {Spectral theory for bounded banded matrices with positive bidiagonal factorization and mixed multiple orthogonal polynomials},
	journal = {Advances in Mathematics},
	volume  = {434},
	year    = {2023},
	pages   = {109313},
	doi     = {10.1016/j.aim.2023.109313}
}

@online{BFM4,
	author      = {Branquinho, Am{\'\i}lcar and Foulqui{\'e}-Moreno, Ana and Ma{\~n}as, Manuel},
	title       = {Banded totally positive matrices and normality for mixed multiple orthogonal polynomials},
	date        = {2024},
	eprint      = {2404.13965},
	eprinttype  = {arxiv},
		eprintclass   = {math.CA},
}

@online{BFM6,
	author        = {Branquinho, Am{\'i}lcar and Foulqui{\'e}-Moreno, Ana and Ma{\~n}as, Manuel},
	title         = {Spectral theory for {M}arkov chains with transition matrix admitting a stochastic bidiagonal factorization},
	eprint        = {2601.10890},
	eprinttype    = {arxiv},
		date        = {2026},
	eprintclass   = {math.PR},
}

@book{Chihara,
	author    = {Chihara, Theodore S.},
	title     = {An Introduction to Orthogonal Polynomials},
	series    = {Dover Books on Mathematics},
	publisher = {Dover Publications},
	address   = {Mineola, NY, USA},
	year      = {2011},
	isbn      = {9780486479293},
	note      = {Reprint of the 1978 edition},
}

@book{Fallat,
	author    = {Fallat, Shaun M. and Johnson, Charles R.},
	title     = {Totally Nonnegative Matrices},
	publisher = {Princeton University Press},
	address   = {Princeton},
	year      = {2011},
	doi       = {10.1515/9781400839018}
}

@book{Gantmacher0,
	author    = {Gantmacher, Felix P.},
	title     = {Matrix Theory},
	volume    = {2},
	publisher = {Chelsea Publishing Company},
	location  = {New York},
	year      = {1974},
}

@book{Gantmacher,
	author    = {Gantmacher, Felix P. and Krein, Mark G.},
	title     = {Oscillation Matrices and Kernels and Small Vibrations of Mechanical Systems},
	edition   = {Revised},
	publisher = {AMS Chelsea Publishing},
	address   = {Providence},
	year      = {2002},
	doi       = {10.1090/chel/345}
}

@book{Ismail,
	author    = {Ismail, Mourad E. H.},
	title     = {Classical and Quantum Orthogonal Polynomials in One Variable},
	series    = {Encyclopedia of Mathematics and its Applications},
	volume    = {98},
	publisher = {Cambridge University Press},
	year      = {2009},
	doi       = {10.1017/CBO9781107325982}
}

@book{Karlin,
	author    = {Karlin, Samuel},
	title     = {Total Positivity},
	publisher = {Stanford University Press},
	address   = {Stanford},
	year      = {1968}
}

@article{andrei_walter,
	author  = {Mart{\'\i}nez-Finkelshtein, Andrei and Van Assche, Walter},
	title   = {What is a multiple orthogonal polynomial?},
	journal = {Notices of the AMS},
	volume  = {63},
	number  = {9},
	year    = {2016},
	pages   = {1029--1031},
	doi     = {10.1090/noti1430}
}

@book{nikishin_sorokin,
	author    = {Nikishin, Evgenii M. and Sorokin, Vladimir N.},
	title     = {Rational Approximations and Orthogonality},
	series    = {Translations of Mathematical Monographs},
	volume    = {92},
	publisher = {American Mathematical Society},
	address   = {Providence, RI},
	year      = {1991},
	doi       = {10.1090/mmono/092}
}

@book{Pinkus,
	author    = {Pinkus, Allan},
	title     = {Totally Positive Matrices},
	series    = {Cambridge Tracts in Mathematics},
	volume    = {181},
	publisher = {Cambridge University Press},
	address   = {Cambridge},
	year      = {2010},
	doi       = {10.1017/CBO9780511691713}
}

\end{document}